\iffalse
\documentclass{mcom-l}
\usepackage{amssymb}
\usepackage{booktabs}
\usepackage[ruled,vlined,linesnumbered]{algorithm2e}
\usepackage{comment}
\newtheorem{thm}{Theorem}
\newtheorem{lem}{Lemma}

\begin{document}

\title[On Grosswald's conjecture on primitive roots]{On Grosswald's conjecture on primitive roots}

\author{Stephen D. Cohen}
\address{School of Mathematics and Statistics, University of Glasgow, Scotland}
\email{Stephen.Cohen@glasgow.ac.uk}
\curraddr{}
\thanks{}

\author{Tom\'{a}s Oliveira e Silva}
\address{Departamento de Electr{\'o}nica, Telecomunica{\c c}{\~o}es e Inform{\'a}tica / IEETA,
  Universidade de Aveiro, Portugal}
\email{tos@ua.pt}
\urladdr{http://www.ieeta.pt/~tos}
\curraddr{}
\thanks{}

\author{Tim Trudgian}
\address{Mathematical Sciences Institute, The Australian National University, ACT 0200, Australia}
\email{timothy.trudgian@anu.edu.au}
\curraddr{}
\thanks{Supported by Australian Research Council DECRA Grant DE120100173}



\dedicatory{}

\begin{abstract}
Grosswald's conjecture is that $g(p)$, the least primitive root modulo $p$, satisfies
$g(p) \leq \sqrt{p} - 2$ for all $p>409$. We make progress towards this conjecture by proving that
$g(p) \leq \sqrt{p} -2$ for all $p>8.01\times 10^{76}$.
\end{abstract}

\maketitle

\else
\documentclass[11pt]{article}
\usepackage{a4wide}

\usepackage{booktabs}
\usepackage[ruled,vlined,linesnumbered]{algorithm2e}
\usepackage{amsthm}
\usepackage{amsmath}
\usepackage{amssymb}
\newtheorem{thm}{Theorem}
\newtheorem{lem}{Lemma}
\newtheorem{cor}{Corollary}

\title{On Grosswald's conjecture on primitive roots}

\author{
  Stephen D. Cohen \\
  School of Mathematics and Statistics, \\
  University of Glasgow, Scotland \\
  Stephen.Cohen@glasgow.ac.uk
\and
  Tom\'{a}s Oliveira e Silva \\
  Departamento de Electr{\'o}nica, Telecomunica{\c c}{\~o}es e Inform{\'a}tica / IEETA \\
  University of Aveiro, Portugal \\
  tos@ua.pt
\and
  Tim Trudgian\footnote{Supported by Australian Research Council DECRA Grant DE120100173.} \\
  Mathematical Sciences Institute \\
  The Australian National University, ACT 0200, Australia \\
  timothy.trudgian@anu.edu.au
}
\date{}

\begin{document}

\maketitle

\begin{abstract}
  \noindent
  Grosswald's conjecture is that $g(p)$, the least primitive root modulo $p$, satisfies
  $g(p) \leq \sqrt{p} - 2$ for all $p>409$. We make progress towards this conjecture by proving
  that $g(p) \leq \sqrt{p} -2$ for all~$409<p< 2.5\times 10^{15}$ and for all~$p>3.67\times 10^{71}$.
\end{abstract}
\textit{AMS Codes: 11L40, 11A07}

\fi

\section{Introduction}

Let $g(p)$ denote the least primitive root of a prime~$p$. Burgess~\cite{Burgess} showed that
$g(p)\ll p^{1/4+\epsilon}$ for any~$\epsilon>0$. This remains the best known bound in general
--- see \cite{MoreeArtin} for an insightful survey of related problems.
Grosswald~\cite{Grosswald81} conjectured that
\begin{equation}\label{gross}
  g(p) < \sqrt{p}-2,
\end{equation}
for all primes~$p>409$. This has implications for the generators of $\Gamma(p)$, the principal congruence subgroup modulo $p$ of the modular group $\Gamma$ --- see \cite[\S 8]{Grosswald81}. Grosswald verified numerically that~(\ref{gross}) is true for all
$409<p\leq 10000$. He also gave an explicit version of Burgess' bound, thereby proving that
$g(p)\leq p^{0.499}$ for all~$p>1+\exp(\exp(24)) \approx 10^{10^{10}}$.

Using computational and theoretical arguments we improve on Grosswald's estimate in the following theorem.
\begin{thm}\label{gross:ext}
  Let $g(p)$ denote the least primitive root modulo~$p$. Then $g(p) \leq \sqrt{p}-2$ for all
  $409 < p < 2.5\times 10^{15}$ and for all $p> 3.67\times 10^{71}$.
\end{thm}
The `gap' in Theorem \ref{gross:ext} between the ranges of $p$ seems difficult to bridge. The
trivial bound $g(p) \leq p$ when combined with the results in Theorem~\ref{gross:ext} gives the following corollary.
\begin{cor}\label{gross:ext2}
  $g(p) \leq 5.19 p^{0.99}$ for all $p$.
\end{cor}
The bound in Corollary~\ref{gross:ext2}, while weak, appears to be the first bound that holds
for all~$p$. The remainder of the paper is organised as follows. In \S\ref{sec:ex} we collect the
necessary results to make Burgess' result explicit. This gives a substantial
improvement on the upper bound $\exp(\exp(24))+1$ given by Grosswald. We introduce a sieving
inequality in \S\ref{sec:sieve} which enables us to reduce this
further. Finally, in \S\ref{sec:comp} we present some computational arguments which complete the
proof of Theorem~\ref{gross:ext}, and present some data on two related problems involving primitive roots.

\section{Explicit versions of Burgess' bounds}\label{sec:ex}

Burgess's bounds on the character sum
\begin{equation*}
  S_{H}(N) = \sum_{m=N+1}^{N+H} \chi(m)
\end{equation*}
were first made explicit by Grosswald [op.\ cit.], and were later refined by Booker~\cite{BookerBurgess}, 
McGown~\cite{McGown}, and, most recently, by Trevi\~{n}o~\cite{Trevino}. The following is Theorem 1.7 in~\cite{Trevino}.
\begin{thm}\label{theorem:M}[Trevi\~{n}o]
  Suppose $\chi$ is a non-principal Dirichlet character modulo~$p$ where $p\geq 10^{20}$.
  Let $N, H\in \mathbb{Z}$ with $H\geq 1$. Fix a positive integer $r\geq 2$. Then there exists a
  computable constant $C(r)$ such that whenever $H\leq 2 p^{1/2 + 1/4r}$ we have
  \begin{equation}\label{gown}
    \bigl|S_{H}(N)\bigr| \leq C(r) H^{1 - 1/r} p^{\frac{r+1}{4r^{2}}} (\log p)^{\frac{1}{2r}}.
  \end{equation}
\end{thm}

We follow Burgess, who, in~\cite[\S 6]{Burgess} considers
\begin{equation}\label{burger}
  f(x) = \frac{\phi(p-1)}{p-1}
  \left\{ 1 + \sum_{d| p-1, d>1} \frac{\mu(d)}{\phi(d)} \sum_{\chi_{d}} \chi_{d}(x)\right\},
\end{equation}
whence it follows that $f(x)=1$ if $x$ is a primitive root, and $f(x)=0$ otherwise. Thus, if
$N(H)$ denotes the number of primitive roots in the interval $N+1 \leq x \leq N+H$ we have
\begin{equation*}
  N(H) = \sum_{x = N+1, x \not\equiv 0 \pmod p}^{N+H} f(x).
\end{equation*}
Hence, if $H<p$ there is at most one $x\in[N+1, N+H]$ with $x \equiv 0 \pmod p$ so that
\begin{equation*}\label{p1}
  \bigg| N(H) - \sum_{x=N+1}^{N+H} f(x) \bigg| \leq 1.
\end{equation*}
We can estimate the sum of $f(x)$ using~(\ref{gown}) with
$H = \Bigl( 1- \frac{2}{p_{0}^{1/2}} \Bigr)p^{\frac{1}{2}}$.
This choice of $H$ guarantees that $H<\sqrt{p}-2$ for~$p>p_{0}$.

Since we need only consider square-free divisors $d$ in the outer sum in~(\ref{burger}), and since
there are $\phi(d)$ characters $\chi_{d}$, we arrive at the following theorem.
\begin{thm}\label{T:3}
  We have $g(p) < \sqrt{p}-2$ for $p>p_{0}$ provided that
  \begin{equation}\label{hunt}
    p^{\frac{r-1}{4r^{2}}} >
    C(r) \biggl(1-\frac{2}{p_{0}^{1/2}}\biggr)^{-1/r} (\log p)^{1/2r}\{ 2^{\omega(p-1)} -1\}.
  \end{equation}
\end{thm}
The exponent on the left side of~(\ref{hunt}) is maximised when~$r=2$. We rearrange~(\ref{hunt})
accordingly to show that we require
\begin{equation}\label{hunt2}
  \frac{p}{(\log p)^{4}} > C(2)^{16} (0.99)^{-8} \left\{ 2^{\omega(p-1)} -1\right\}^{16},
  \qquad(p> 10^{20}).
\end{equation}
With $C(2)=3.5751$ as in \cite[Table 3]{Trevino}, we see that (\ref{hunt2}) is true whenever
$\omega(p-1) \geq 17984$. Hence we need only consider $\omega(p-1) \leq 17983$. Solving for $p$ in (\ref{hunt2}) we find we need only consider
$p<10^{86650}$, which is much less than $10^{10^{10}}$. We reduce this upper bound substantially by
introducing a sieving inequality in the next section.

\section{A sieving inequality}\label{sec:sieve}

Let $e$ be an even divisor of~$p-1$. Let $\mathrm{Rad}(n)$ denote the product of the distinct
prime divisors of~$n$. If $\mathrm{Rad}(e) =\mathrm{Rad}(p-1)$, then set $s=0$ and~$\delta =1$.
Otherwise, if $\mathrm{Rad}(e) < \mathrm{Rad}(p-1)$, let $p_1, \ldots, p_s$, $s \geq 1$, be the
primes dividing $p-1$ but not $e$ and set $\delta=1-\sum_{i=1}^s p_i^{-1}$. In practice, it is
essential to choose $e$ so that~$\delta >0$.

Again let $e$ be an even divisor of~$p-1$. An integer $x$ (indivisible by $p$) will be called
\emph{$e$-free} if, for any divisor $d$ of $e$, (with $d>1$), the congruence $x \equiv y^d \pmod p$ is
insoluble. With this terminology, a primitive root is $(p-1)$-free. Given $N$ and $H$ let $N_e(H)$
be the number of integers $x$ in the range $N+1 \leq x \leq N+H$ that are indivisible by $p$ and
such that $x$ is $e$-free.
\begin{lem}\label{sieve}
  Suppose $e$ is an even divisor of $p-1$. Then, in the above notation,
  \begin{equation}\label{sieveeq1}
    N_{p-1}(H) \geq \sum_{i=1}^sN_{p_ie}(H)-(s-1)N_e(H).
  \end{equation}
  Hence
  \begin{equation}\label{sieveeq2}
    N_{p-1}(H) \geq \sum_{i=1}^s \bigl[ N_{p_ie}(H)-\theta(p_i)N_e(H) \bigr] + \delta N_e(H).
  \end{equation}
\end{lem}
\begin{proof}
For a given $e$-free integer $x$, the right side of~(\ref{sieveeq1}) contributes 1 if $x$ is
additionally $p_{i}$-free, and otherwise contributes a non-positive
quantity.
We deduce~(\ref{sieveeq2}) by rearranging~(\ref{sieveeq1}) bearing in mind the definitions for
$\theta(p_{i})$ and~$\delta$.
\end{proof}

Given the divisor $e$ of $p-1$, we extend the definition of $f(x)$ to $f_e(x)$, where
\begin{equation*}\label{fries}
  f_e(x) = \theta(e)
  \left\{ 1 + \sum_{d| e, d>1} \frac{\mu(d)}{\phi(d)} \sum_{\chi_{d}} \chi_{d}(x)\right\},
\end{equation*}
and where $\theta(e)= \frac{\phi(e)}{e}$. Hence $f_e(x) = 1$ if $x$ is $e$-free, and $f_e(x) =0$
otherwise. Thus,
\begin{equation*}
  N_e(H) = \sum_{x = N+1, x \not\equiv 0 \pmod p}^{N+H} f_e(x).
\end{equation*}
It follows from Theorem~\ref{gown} that, under the constraints of that theorem,
\begin{equation}\label{cat}
  N_e(H) \geq  \theta(e) \Bigl( H - \bigl( W(e)-1 \bigr) C(r) H^{1-1/r} p^{\frac{r+1}{4r^{2}}}
  (\log p)^{\frac{1}{2r}} \Bigr),
\end{equation}
where $W(e)= 2^{\omega(e)} $ is the number of square-free divisors of~$e$.

Similarly, for any prime divisor $l$ of $p-1$ not dividing $e$,
\begin{equation}\label{dog}
  \left|N_{le}(H) - \left(1- \frac{1}{l}\right)N_e(H)\right| \leq
  \theta(e)W(e)C(r) H^{1 - 1/r} p^{\frac{r+1}{4r^{2}}} (\log p)^{\frac{1}{2r}},
\end{equation}
where the factor $W(e)$ arises from the expression $W(le)-W(e)$.

Now apply~(\ref{cat}) and~(\ref{dog}) to~(\ref{sieveeq2}) to obtain
\begin{equation*}
  N_{p-1}(H) \geq \delta \theta(e)H- C(r) H^{1 - 1/r} p^{\frac{r+1}{4r^{2}}}
  (\log p)^{\frac{1}{2r}}W(e)\left(\delta +\sum_{i=1}^s\frac{1}{p_i}\right).
\end{equation*}
Since $\sum_{i=1}^s\frac{1}{p_i}=s-1+\delta$, this yields
\begin{equation}\label{pie}
  N_{p-1}(H) \geq \delta\theta(e) \left\{H- W(e)C(r) H^{1 - 1/r} p^{\frac{r+1}{4r^{2}}}
  (\log p)^{\frac{1}{2r}}\left(\frac{s-1}{\delta}+2 \right)\right\}.
\end{equation}

As in \S \ref{sec:ex}, we take $H = \Bigl( 1- \frac{2}{p_{0}^{1/2}} \Bigr)p^{\frac{1}{2}}$ and $r=2$
in~(\ref{pie}). This proves the following refinement of Theorem \ref{T:3}.
\begin{thm}\label{thm:sieve:final}
  Let $e$ be an even divisor of $p-1$ and $s, \delta$ as in Lemma $\ref{sieve}$ with~$\delta>0$.
  We have $g(p) < \sqrt{p} - 2$ for $p>p_{0}$ provided that
  \begin{equation}\label{horse}
    \theta (e)\left( 1- \frac{2}{p_{0}^{\frac{1}{2}}}\right) p^{\frac{1}{2}}
    \left\{ 1 - C(2) \left( 1- \frac{2}{p_{0}^{\frac{1}{2}}}\right)^{-\frac{1}{2}}
    (\log p)^{\frac{1}{4}} \left( \frac{s-1}{\delta} + 2\right)W(e) p^{-\frac{1}{16}}\right\} >0.
  \end{equation}
\end{thm}
We can rearrange~(\ref{horse}) to show that our criterion becomes
\begin{equation}\label{cow}
  \frac{p}{\log^{4} p} > C(2)^{16} \left( 1- \frac{2}{p_{0}^{\frac{1}{2}}}\right)^{-8}
  \left\{ \left( \frac{s-1}{\delta} +2\right) 2^{n-s}\right\}^{16}.
\end{equation}
We consider~(\ref{cow}) for $\omega(p-1) = n \leq 17983$. By making the choice of $s$ for $n$
given in Table~\ref{tab:rat} we verify~(\ref{cow}) for
all~$n\geq 42.$

\begin{table}[ht]
  \caption{Choices of $s$ for various ranges of $\omega(p-1)=n$ such that (\ref{cow}) holds.}
  \centering
  \begin{tabular}{cc}
    \hline\hline
    Range of $\omega(p-1)=n$ & $s$ \\[0.5ex]\hline
    $[800,17983]$ & $750$ \\
    $[ 400, 799]$ &  $300$ \\
    $[ 200,  399]$ &  $180$ \\
    $[ 105,  199]$ &  $105$ \\
    $[ 72,  104]$ &   $68$ \\
    $[  55,  71]$ &   $52$ \\
    $[  47,   54]$ &   $44$ \\
    $[  43, 46]$ &   $40$ \\
              $42$ &   $38$ \\
    \hline\hline
  \end{tabular}
  \label{tab:rat}
\end{table}

We are left with those $p$ satisfying $\omega(p-1) \leq 41$.
When $\omega(p-1)=n=41$ we choose $s=37$ to minimise the right-side of~(\ref{cow}). This shows
that Grosswald's conjecture is satisfied provided that
\begin{equation}\label{qmax}
  \frac{p}{\log^{4} p} > 4.97\times 10^{62}.
\end{equation}
Solving~(\ref{qmax}) for $p$ gives $p>3.67\times 10^{71}$. It is tempting to try to remove the
$\omega(p-1)=41$ case by enumerating possible primes as in~\cite{COT1}. Since
$p-1 > p_{1} \cdots p_{41}$ we seek the number of solutions of
\begin{equation}\label{number}
  2.98\times 10^{70}\leq p \leq 3.67\times 10^{71},
  \qquad p \; \textrm{prime},
  \quad \omega(p-1) = 41.
\end{equation}
A quick computer check shows that there are 329 different primes that could appear in the
factorisation of~$p-1$. While it may be possible to enumerate all such  products satisfying (\ref{number}), this would, at best, eliminate the $n=41$ case only. We have not pursued such an enumeration.

\section{Computational results}\label{sec:comp}

The computational part of Theorem~\ref{gross:ext} was proved in the following way. The interval
$[2,10^{15}]$ was subdivided into consecutive sub-intervals of manageable size (each with $2^{20}$
integers). An efficient segmented Eratosthenes sieve (see \cite{Bays-1977-2-SSEPAP} and 
\cite[\S 1.1]{Oliveira.e.Silva-2013-3-EVGC}) was then used to identify all  primes in each
interval. For each prime $p$ that was found, a second Eratosthenes sieve, modified to yield
complete factorizations~\cite[\S 3.2.4]{Crandall-2002-2-PNCP}, was used to find the
factorization of $(p-1)/2$. Since the least primitive root modulo~$p$ cannot be of the form $a^b$
with $a>0$ and $b>1$, i.e., it cannot be a perfect power, the integers $2,3,5,6,7,10,\ldots,$ were
tried one at a time until a primitive root was found. 

With $c$ as a candidate primitive root, the
first test was to check if $c^{(p-1)/2}\equiv -1 \pmod p$. This was efficiently done
using the quadratic reciprocity law data from known tables. If this test failed the next $c$
candidate was tried. Otherwise, for each odd prime factor $q$ of $(p-1)/2$ it was checked whether
$c^{(p-1)/q}\not\equiv 1 \pmod p$. The next $c$ candidate was tried if one of these tests failed.
These tests were efficiently done by performing all modular arithmetic using the Montgomery
method~\cite{Montgomery-1985-1-MMWTD}. Since the ``probability'' of failure of an individual test
is $1/q$, the odd factors $q$ were sorted in increasing order before performing these tests.
Note that $g(p)$ is equal to the first $c$ that passes all tests.

Instead of checking~(\ref{gross}) directly for each prime up to $2.5\times 10^{15}$, the
record-holder values of $g(p)$, i.e., values of $g(p)$ such that $g(p')<g(p)$ for all $p'<p$, were
computed, as these are of independent interest~\cite{Bach-1997-1-CSPPR} and can be used to
check~(\ref{gross}) indirectly. The computation required a total time of about $3$ one-core
years, and took about one month to finish on nine computers (each with $4$ cores) of one computer
lab of the Electronics, Telecommunications, and Informatics Department of the University of
Aveiro. Table~\ref{tab:precords} presents all $g(p)$ record-holders that were found up
to~$2.5\times 10^{15}$. It extends and corrects one entry of Table~2
of~\cite{Paszkiewicz-2002-2-DPNLPR}, which is a summary of computations up to~$4\times 10^{10}$.

\begin{table}[ht]
  \caption{$g(p)$ record-holders with $p<2.5\times 10^{15}$.}
  \centering
  \begin{tabular}{rr@{$\quad\qquad$}rr@{$\quad\qquad$}rr}
    \hline\hline
    $g(p)$ &   $p$ & $g(p)$ &        $p$ & $g(p)$ &              $p$ \\[0.5ex]\hline
         2 &     3 &     69 &     110881 &    179 &       6064561441 \\
         3 &     7 &     73 &     760321 &    194 &       7111268641 \\
         5 &    23 &     94 &    5109721 &    197 &       9470788801 \\
         6 &    41 &     97 &   17551561 &    227 &      28725635761 \\
         7 &    71 &    101 &   29418841 &    229 &     108709927561 \\
        19 &   191 &    107 &   33358081 &    263 &     386681163961 \\
        21 &   409 &    111 &   45024841 &    281 &    1990614824641 \\
        23 &  2161 &    113 &   90441961 &    293 &   44384069747161 \\
        31 &  5881 &    127 &  184254841 &    335 &   89637484042681 \\
        37 & 36721 &    137 &  324013369 &    347 &  358973066123281 \\
        38 & 55441 &    151 &  831143041 &    359 & 2069304073407481 \\
        44 & 71761 &    164 & 1685283601                             \\
    \hline\hline
  \end{tabular}
  \label{tab:precords}
\end{table}

The largest $g(p)$ record-holder in Table~\ref{tab:precords} that does not satisfy~(\ref{gross})
is~$21$, corresponding to $p=409$. Thus, up to $2.5\times 10^{15}$, the largest $p$ for
which~(\ref{gross}) is possibly false satisfies $\sqrt{p}-2<21$, i.e., $p<529$. It turns out, as
already verified by Grosswald, that the last failure of~(\ref{gross}) occurs for $p=409$.

An analysis similar to the one described above was also performed for least prime primitive
roots~$\hat{g}(p)$, and for least negative primitive roots~$h(p)$. The least negative primitive
root modulo~$p$ is equal to the negative integer, least in absolute value, that is a primitive
root modulo~$p$. It cannot be of the form $-a^b$ with $a>0$ and $b>2$, and is equal to $-g(p)$ if
$p\equiv 1\pmod 4$. It was found that $\hat{g}(p)<\sqrt{p}-2$ for $2791<p<2.5\times 10^{15}$, and
that $-h(p)<\sqrt{p}-2$ for $409<p<10^{15}$. We remark that little is known about either $\hat{g}(p)$ or $h(p)$ --- the reader may consult \cite{Martin} for more details.

\section{Conclusion}\label{sec:conclusion}

It appears difficult to resolve completely Grosswald's conjecture. Table 3 in \cite{Trevino} indicates that one may hope to reduce the size of $C(2)$ further by taking a larger value of $p_{0}$. However, this appears at present not to give an improvement for our purposes.

An alternative approach is to
use a smoothed version of Burgess' bounds, in the same way that a smoothed P\'{o}lya--Vinogradov
inequality was used in~\cite{LPS2010}. 

\section{Acknowledgements}
Some of this work was completed when the third author visited the first author. This visit was supported by the Royal Society of Edinburgh and the Edinburgh Mathematical Society: the authors are grateful for this support and for the hospitality of the School of Mathematics and Statistics at the University of Glasgow.


\begin{thebibliography}{10}

\bibitem{Bach-1997-1-CSPPR}
E.~Bach.
\newblock Comments on search procedures for primitive roots.
\newblock {\em Math. Comp.}, 66(220):1719--1727, 1997.
 
\bibitem{Bays-1977-2-SSEPAP}
C.~Bays and R.~H. Hudson.
\newblock The segmented sieve of {E}ratosthenes and primes in arithmetic
  progressions to $10^{12}$.
\newblock {\em Nordisk Tidskr. Informationsbehandling (BIT)},
  17(2):121--127, June 1977.

\bibitem{BookerBurgess}
A.~R. Booker.
\newblock Quadratic class numbers and character sums.
\newblock {\em Math. Comp.}, 75(255):1481--1492, 2006.

\bibitem{Burgess}
D.~A. Burgess.
\newblock On character sums and primitive roots.
\newblock {\em Proc. London Math. Soc.}, 12(3):179--192, 1962.

\bibitem{COT1}
S.~D. Cohen, T.~Oliveira~e Silva, and T.~S. Trudgian.
\newblock A proof of the conjecture of {C}ohen and {M}ullen on sums of
  primitive roots.
\newblock {\em Math. Comp.}
\newblock To appear.

\bibitem{Crandall-2002-2-PNCP}
R.~Crandall and C.~Pomerance.
\newblock {\em Prime Numbers: A Computational Perspective}.
\newblock Springer, New York, 2002 (second edition).

\bibitem{Grosswald81}
E.~Grosswald.
\newblock On {B}urgess' bound for primitive roots modulo primes and an
  application to {$\Gamma(p)$}.
\newblock {\em Amer. J. Math.}, 103(6):1171--1183, 1981.

\bibitem{LPS2010}
M.~Levin, C.~Pomerance, and K.~Soundararajan.
\newblock Fixed points for discrete logarithms.
\newblock {\em Lecture Notes in Comput. Sci.}, 6197:6--15, 2010.

\bibitem{Martin}
G.~Martin.
\newblock The least prime primitive root and the shifted sieve.
\newblock {\em Acta Arith.}, 80(3):277--288, 1997.


\bibitem{McGown}
K.~J. McGown.
\newblock Norm-{E}uclidean cyclic fields of prime degree.
\newblock {\em Int. J. Number Theory}, 8(1):227--254, 2012.


\bibitem{Montgomery-1985-1-MMWTD}
P.~L.~Montgomery.
\newblock Modular multiplication without trial division.
\newblock {\em Math. Comp.}, 44(170):519--521, 1985.

\bibitem{MoreeArtin}
P.~Moree.
\newblock Artin's primitive root conjecture --- a survey.
\newblock {\em Integers}, 12(6):1305--1416, 2012.

\bibitem{Oliveira.e.Silva-2013-3-EVGC}
T.~Oliveira~e Silva, S.~Herzog, and S.~Pardi.
\newblock {E}mpirical verification of the even {G}oldbach conjecture and
  computation of prime gaps up to $4\cdot 10^{18}$.
\newblock {\em Math. Comp.}, 83(288):2033--2060, 2014.
\newblock Published electronically on November 18, 2013.

\bibitem{Paszkiewicz-2002-2-DPNLPR}
A.~Paszkiewicz and A.~Schinzel.
\newblock Numerical calculation of the density of prime numbers with a given
  least primitive root.
\newblock {\em Math. Comp.}, 71(240):1781--1797, 2002.

\bibitem{Trevino}
E.~Trevi\~{n}o.
\newblock The {B}urgess inequality and the least $k$-th power non-residue.
\newblock {\em Int. J. Number Theory}, to appear, 2015.

\end{thebibliography}
\end{document}